\documentclass[11pt,a4paper,fleqn]{article}
\usepackage{amsmath}
\usepackage{amssymb} 
\usepackage{latexsym} 
\usepackage{amscd} 
\usepackage{enumerate} 
\usepackage{mathrsfs} 
\usepackage[latin2]{inputenc}
\newtheorem{theorem}{Theorem}

\newtheorem{corollary}[theorem]{Corollary}
\newtheorem{definition}[theorem]{Definition}

\newtheorem{lemma}[theorem]{Lemma}
\newtheorem{proposition}[theorem]{Proposition}

\newtheorem{remark}[theorem]{Remark}
\newenvironment{proof}{{\bf Proof.}}{\hfill$\rule{1ex}{1ex}$\par\medskip}

\begin{document}
\topmargin=-0.0in
\textheight= 8.5in
\textwidth=5.5in
\oddsidemargin=-0.0in
\newcommand{\bbbr}{\mbox{$\mathbb R$}}%{\mbox{$I\!\!R$}}
\newcommand{\bbbn}{\mbox{$\mathbb N$}}
\newcommand{\bbbc}{\mbox{$\mathbb C$}}
\newcommand{\R}{\mbox{$\mathbb R$}}
\def\n{\nabla}
\def\I{\mathfrak I}
\def\F{\mathcal F}
\def\g{\mathfrak{g}}
\def\h{\mathfrak{h}}
\def\m{\mathfrak{m}}
\def\ad{\rm ad}
\def\Ad{\rm Ad}
\def\E{\mathcal E}
\def\pr{\mathfrak{h}}
\def\TM{{\mathcal T}M}
\def\P{{{\mathcal R}}}
\newcommand{\bt}{\begin{theorem}}
\newcommand{\et}{\end{theorem}}
\newcommand{\bd}{\begin{definition}}
\newcommand{\ed}{\end{definition}}
\newcommand{\bc}{\begin{corollary}}
\newcommand{\ec}{\end{corollary}}
\newcommand{\bs}{\begin{proposition}}
\newcommand{\es}{\end{proposition}}
\newcommand{\bl}{\begin{lemma}}
\newcommand{\el}{\end{lemma}}
\newcommand{\bp}{\begin{proof}\;}
\newcommand{\ep}{\end{proof}}
\newcommand{\be}{\begin{equation}}
\newcommand{\eeq}{\end{equation}}
\newcommand{\ul}{\underline}

\title{Finsler manifolds with non-Riemannian holonomy}

\author{Zolt\'an Muzsnay and P\'eter T. Nagy}

\date{Institute of Mathematics, University of Debrecen\\
  H-4010 Debrecen, Hungary, P.O.B. 12
  \\
  \bigskip
  {\it E-mail}: {\tt {}muzsnay@math.unideb.hu}, {\tt
    {}nagypeti@math.unideb.hu}}
\maketitle

\footnotetext{This research was supported by the Hungarian Scientific
  Research Fund (OTKA) Grant K 67617.}

%%%%%%%%%%%%%%%%%%%%%%%%%%%%%%%%%%%%%%%%%%%%%%%%%%%%%%%%%%%%%%%%%%%%%

\begin{abstract}
  The aim of this paper is to show that holonomy properties of Finsler
  manifolds can be very different from those of Riemannian
  manifolds. We prove that the holonomy group of a positive definite
  non-Riemannian Finsler manifold of non-zero constant curvature with
  dimension $>2$ cannot be a compact Lie group.  Hence this holonomy
  group does not occur as the holonomy group of any Riemannian
  manifold.  In addition, we provide an example of left invariant
  Finsler metric on the Heisenberg group, so that its holonomy group
  is not a (finite dimensional) Lie group. These results give a
  positive answer to the following problem formulated by S. S. Chern
  and Z. Shen: {\em Is there a Finsler manifold whose holonomy group
    is not the holonomy group of any Riemannian manifold?}
\end{abstract}

%%%%%%%%%%%%%%%%%%%%%%%%%%%%%%%%%%%%%%%%%%%%%%%%%%%%%%%%%%%%

\section{Introduction}

The notion of the holonomy group of a Riemannian manifold can be
generalized very naturally for a Finsler manifold (cf. e.g. S. S.
Chern and Z. Shen, \cite{ChSh}, Chapter 4): it is the group at a point
$x$ generated by the canonical homogeneous (nonlinear) parallel
translations along all loops emanated from $x$. Until now the holonomy
groups of non-Riemannian Finsler manifolds have been described only in
special cases: for Berwald manifolds there exist Riemannian metrics
with the same holonomy group (cf. Z. I. Szab\'o, \cite{Sza}), for
positive definite Landsberg manifolds the holonomy groups are compact
Lie groups consisting of isometries of the indicatrix with respect to
an induced Riemannian metric (cf. L. Kozma, \cite{Koz1},
\cite{Koz2}). A thorough study of the holonomy group of homogeneous
(nonlinear) connections was initiated by W. Barthel in his basic work
\cite{Bar} in 1963; he gave a construction for a holonomy algebra of
vector fields on the tangent space.  A general setting for the study
of infinite dimensional holonomy groups and holonomy algebras of
nonlinear connections was initiated by P. Michor in
\cite{Mic1}. However the introduced holonomy algebras could not be
used to estimate the dimension of the holonomy group since their
tangential properties to the holonomy group were not clarified.
\\[1ex]\indent The aim of this paper is to show that holonomy
properties of Finsler manifolds can be very different from those of
Riemannian manifolds. We prove that if the holonomy group of a
non-Riemannian Finsler manifold of non-zero constant curvature with
dimension $n>2$ is a (finite dimensional) Lie group then its dimension
is strictly greater than the dimension of the orthogonal group acting
on the tangent space and hence it can not be a compact Lie group.  An
estimate for the dimension of the holonomy group will be obtained by
investigation of a Lie algebra of tangent vector fields on the
indicatrix, algebraically generated by curvature vector fields of the
Finsler manifold. We call this Lie algebra the \textit{curvature
  algebra} and prove that its elements are tangent to one-parameter
families of diffeomorphisms contained in the holonomy group. For
non-Riemannian Finsler manifolds of constant curvature $\ne 0$ with
dimension $n>2$ we construct more than $\frac{n(n-1)}{2}$ linearly
independent curvature vector fields. \\[1ex]\indent In addition, we
provide an example of a left invariant singular (non $y$-global) Finsler metric of
Berwald-Mo\'or-type on the Heisenberg group which has infinite
dimensional curvature algebra and hence its holonomy is not a (finite
dimensional) Lie group.  These results give a positive answer to the
following problem formulated by S. S. Chern and Z. Shen in \cite{ChSh}
(p. 85): {\em Is there a Finsler manifold whose holonomy group is not
  the holonomy group of any Riemannian manifold?} This question is
contained also in the list of open problems in Finsler geometry by
Z. Shen \cite{Shen2}, (March 8, 2009, Problem 34).

\section{Preliminaries}
\subsubsection*{Finsler manifold and its canonical connection}

A \emph{Minkowski functional} on a vector space $V$ is a continuous
function $\mathcal F$, positively homogeneous of degree two,
i.e.~$\mathcal F(\lambda y)=\lambda^2 \mathcal F(y)$, smooth on $\hat
V :=V\setminus \{0\}$, and for any $y\in \hat V$ the symmetric
bilinear form $g_y\colon V \times V \to \mathbb R$ defined by
\begin{displaymath}
  g_y \colon (u,v)\ \mapsto \ g_{ij}(y)u^iv^j=\frac{1}{2}
  \frac{\partial^2 \mathcal F(y+su+tv)}{\partial s\,\partial
    t}\Big|_{t=s=0}
\end{displaymath}
is non-degenerate.  If $g_y$ is positive definite for any $y\in \hat
V$ then $\mathcal F$ is said positive definite and $(V,\mathcal F)$ is
called \emph{positive definite Minkowski space}.  A Minkowski
functional $\mathcal F$ is called \emph{semi-Euclidean} if there
exists a symmetric bilinear form $\langle \,, \, \rangle$ on $V$ such
that $g_y(u,v)=\langle u, v \rangle$ for any $y\in \hat V$ and $u,v\in
V$. A semi-Euclidean positive definite Minkowski functional is called
\emph{Euclidean}.  \\[1ex]\indent A \emph{Finsler manifold} is a pair
$(M,\mathcal F)$ where $M$ is an $n$-dimensional manifold and
$\F\colon TM \to \R$ is a function (called \emph{Finsler metric},
cf.~\cite{Shen1}) defined on the tangent bundle of $M$, smooth on
$\hat T M := TM\setminus\! \{0\}$ and its restriction ${\mathcal
  F}_x={\mathcal F}|_{_{T_xM}}$ is a Minkowski functional on $T_xM$
for all $x\in M$. If the restriction ${\mathcal F}_x={\mathcal
  F}|_{_{T_xM}}$ of the Finsler metric $\F\colon TM \to \R$ is
positive definite on $T_xM$ for all $x\in M$ then $(M, \mathcal F)$ is
called \emph{positive definite} Finsler manifold. A point $x\in M$ is
called \emph{(semi-)Riemannian} if the Minkowski functional $\F_x$ is
(semi-)Euclidean.
\\[1ex]
\indent We remark that in many applications the metric $\mathcal F$ is
smooth only on an open cone $\mathcal C M \! \subset\! TM \!
\setminus\! \{0\}$, where $\mathcal C M \! =\! \cup_{x\in M} \mathcal
C_x M$ is a fiber bundle over $M$ such that each $\mathcal C_xM$ is an
open cone in $T_xM \!  \setminus \! \{0\}$. In such case $(M, \mathcal
F)$ is called \emph{singular} (or \emph{non $y$-global}) Finsler space
(cf.~\cite{Shen1}).
\\[1ex]
\indent \emph{Geodesics} of Finsler manifolds are determined by a
system of 2nd order ordinary differential equation:
 \begin{displaymath}
   \ddot{x}^i + 2 G^i(x,\dot x)=0, \quad i = 1,...,n
 \end{displaymath}
where $G^i(x,\dot x)$ are locally given by
\begin{displaymath}
  G^i(x,y):= \frac{1}{4}g^{il}(x,y)\Big(2\frac{\partial
    g_{jl}}{\partial x^k}(x,y) -\frac{\partial g_{jk}}{\partial
    x^l}(x,y) \Big) y^jy^k.
\end{displaymath}
The associated homogeneous (nonlinear) parallel translation can be 
defined as follows: a vector field $X(t)=X^i(t)\frac{\partial}{\partial x^i}$ 
along a curve $c(t)$ is said to be parallel if it satisfies
\begin{equation} 
  \label{eq:nabla}
  \nabla_{\dot c} X (t):=\Big(\frac{d X^i(t)}{d t}+ \Gamma^i_j(c(t),X(t))\dot c^j(t)
  \Big)\frac{\partial}{\partial x^i},
\end{equation}
where $\Gamma^i_j=\frac{\partial G^i}{\partial y^j}$.

\subsubsection*{Horizontal distribution, curvature}

The geometric structure associated to $\nabla$ can be given on $TM$ in
terms of the horizontal distribution.  Let ${\mathcal V}TM \!
\subset\! TTM$ denote the vertical distribution on $TM$, ${\mathcal
  V}_yTM:= \mathrm{Ker} \, \pi_{*,y}$.  The horizontal distribution
${\mathcal H}TM \!  \subset\! TTM$ associated to (\ref{eq:nabla}) is
locally generated by the vector fields
\begin{equation}
  \label{eq:lift}
  l_{(x,y)}\Big(\frac{\partial}{\partial x^i}\Big):=
  \frac{\partial}{\partial x^i}
  +\Gamma_i^k(x,y)\frac{\partial}{\partial y^k},\quad i=1,\dots ,n.
\end{equation}
For any $y\in TM$ we have the decomposition $T_yTM = {\mathcal H}_yTM \oplus {\mathcal
  V}_yTM$. The projectors corresponding to this decomposition will be
denoted by $h_y$ and $v_y$.  The isomorphism $l_{(x,y)}:T_xM \to
{\mathcal H}_yTM$ defined by the formula (\ref{eq:lift}) is called
\emph{horizontal lift}. Then a vector field $X(t)$ along a curve
$c(t)$ is parallel if and only if it is a solution of the differential
equation
\begin{equation}
  \label{eq:par_lift}
  \frac{d}{dt}X(t)=l_{X(t)}(\dot c(t)).
\end{equation}
The \emph{curvature tensor} field characterizes the integrability of
the horizontal distribution:
\begin{equation}
  \label{eq:R_1}
  R_{(x,y)}(\xi, \eta):=v[h \xi, h \eta], \qquad \xi, \eta\in
  T_{(x,y)}TM.
\end{equation}
 Using local coordinate system we have
\begin{displaymath}
  R_{(x,y)} = \left( \frac{\partial \Gamma^k_i}{\partial x^j} -
    \frac{\partial \Gamma^k_j}{\partial x^i} + \Gamma_i^m
    \frac{\partial \Gamma^k_j}{\partial y^m} - \Gamma_j^m
    \frac{\partial \Gamma^k_i}{\partial y^m} \right)dx^i\otimes dx^j
  \otimes\frac{\partial}{\partial y^k}.
\end{displaymath}
The manifold is called of constant curvature $c\in\R$, if for any
$x\in M$ the local expression of the curvature is
\begin{equation}\label{gorb}
  R_{(x,y)} = c\left(\delta_i^kg_{jm}(y)y^m - \delta_j^kg_{im}(y)y^m
  \right)dx^i\otimes dx^j \otimes \frac{\partial}{\partial y^k}.
\end{equation}
In this case the flag curvature of the Finsler manifold
(cf.~\cite{ChSh}, Section 2.1 pp. 43-46) does not depend either on the
point or on the 2-flag.

\subsubsection*{Indicatrix bundle} 

Let $(M,\mathcal F)$ be an $n$-dimensional Finsler manifold.  The
\emph{indicatrix} $\I_xM$ at $x \in M$ is a hypersurface of $T_xM$
defined by
\begin{displaymath}
  \I_xM:= \{y \in T_xM ; \ \mathcal F(y) = \pm 1\}.
\end{displaymath}
If the Finsler manifold $(M,\mathcal F)$ is positive definite then the
indicatrix ${\I}_xM$ is a compact hypersurface in the tangent space
$T_xM$, diffeomorphic to the standard $(n-1)$-sphere. In this case the
group ${\mathsf {Diff}}({\I}_xM)$ of all smooth diffeomorphisms of
${\I}_xM$ is a regular infinite dimensional Lie group modeled on the
vector space ${\mathfrak X}({\I}_xM)$ of smooth vector fields on
${\I}_xM$.  The Lie algebra of the infinite dimensional Lie group
${\mathsf{Diff}}(\I_xM)$ is the vector space ${\mathfrak X}({\I}_xM)$,
equipped with the negative of the usual Lie bracket, (c.f. A. Kriegl
and P. W. Michor \cite{KrMi}, Section 43).
\\
Let $(\I M,\pi,M)$ denote the \emph {indicatrix bundle} of
$(M,\mathcal F)$ and $i:\I M \hookrightarrow T M$ the natural
embedding of the indicatrix bundle into the tangent bundle
$(TM,\pi,M)$.

\subsubsection*{Parallel translation and holonomy}

Let $(M, \F)$ be a Finsler manifold. The \emph{parallel translation}
$\tau_{c}:T_{c(0)}M\to T_{c(1)}M$ along a curve $c:[0,1]\to \R$ is
defined by vector fields $X(t)$ along $c(t)$ which are solutions of
the differential equation (\ref{eq:nabla}). Since
$\tau_{c}:T_{c(0)}M\to T_{c(1)}M$ is a differentiable map between
$\hat T_{c(0)}M$ and $\hat T_{c(1)}M$ preserving the value of the
Finsler metric, it induces a map
\begin{equation}
  \label{eq:tau_ind}
  \tau^\I_{c}\colon \I_{c(0)}M \longrightarrow \I_{c(1)}M
\end{equation}
between the indicatrices. 
\begin{definition} 
  The \emph {holonomy group} $\mathsf{Hol}(x)$ of a Finsler space $(M,
  \F)$ at $x\in M$ is the subgroup of the group of diffeomorphisms
  ${\mathsf{Diff}}({\I}_xM)$ of the indicatrix ${\I}_xM$ determined by
  parallel translation of ${\I}_xM$ along piece-wise differentiable
  closed curves initiated at the point $x\in M$.
\end{definition}
We note that the holonomy group $\mathsf{Hol}(x)$ is a topological
subgroup of the regular infinite dimensional Lie group ${\mathsf
{Diff}}(\I_xM)$, but its differentiable structure is not known in general.

\section{Tangent Lie algebras to subgroups of $\mathsf{Diff}^{\infty}(M)$}

Let $H$ be a subgroup of the diffeomorphism group
$\mathsf{Diff}^{\infty}(M)$ of a differentiable manifold $M$ and let
${\mathfrak X}^{\infty}(M)$ be the Lie algebra of smooth vector fields
on $M$.  
\bd 
A vector field $X\!\in\! {\mathfrak X}^{\infty}(M)$ is
called \emph{strongly tangent} to $H$, if there exists a ${\mathcal
  C}^{\infty}$-differentiable $k$-parameter family
$\{\phi_{(t_1,\dots,t_k)}\in H\}_{t_i\in (-\varepsilon,\varepsilon)}$
of diffeomorphisms such that
\begin{enumerate}[(i)]
  \itemsep=1pt
\item $\phi_{(t_1,\dots,t_k)}=\mathsf{Id}$, if $t_j=0$ for some $1
  \leq j\leq k;$
\item $\frac{\partial^k\phi_{(t_1,\dots,t_k)}}{\partial
    t_1\cdots\partial t_k}\big|_{(t_1,\dots,t_k)=(0,\dots,0)}=X$.
\end{enumerate}
A vector field $X\!\in\! {\mathfrak X}^{\infty}(M)$ is called
\emph{tangent} to $H$, if there exists a ${\mathcal
  C}^1$-differentiable $1$-parameter family $\{\phi_t\in H\}_{t\in
  (-\varepsilon,\varepsilon)}$ of diffeomorphisms of $M$ such that
$\phi_{0}=\mathsf{Id}$ and \( \frac{\partial\phi_t}{\partial
  t}\big|_{t=0}=X.\)
\\
A Lie subalgebra $\mathfrak g$ of ${\mathfrak X}^{\infty}(M)$ is
called \emph{tangent} to $H$, if all elements of $\mathfrak g$ are
tangent vector fields to $H$.
\ed 
\bt 
\label{liealg} 
Let $\mathcal V$ be a set of vector fields strongly tangent to the
subgroup $H$ of $\mathsf{Diff}^{\infty}(M)$. The Lie subalgebra
$\mathfrak v$ of ${\mathfrak X}^{\infty}(M)$ generated by $\mathcal V$
is tangent to $H$.  \et
\begin{proof} 
  First, we investigate some properties of vector fields strongly
  tangent to the group $H$.
  \begin{lemma} 
    \label{-1} 
    Let \(\{\psi_{(t_1,...,t_h)}\in
    \mathsf{Diff}^{\infty}(U)\}_{t_i\in (-\varepsilon,
      \varepsilon)}\)\, be a ${\mathcal C}^{\infty}$-differentiable
    $h$-para\-meter family of (local) diffeomorphisms on a neighbourhood
    $U\subset{\mathbb R}^n$, satisfying
    $\psi_{(t_1,\dots,t_h)}=\mathsf{Id}$, if $t_j=0$ for some $1 \leq
    j\leq h$. Then
\begin{enumerate}
\item[\em{(i)}]\label{item_1}
  $\displaystyle{\frac{\partial^{i_1+...+i_h}\psi_{(t_1,...,t_h)}}{\partial
      t_1^{i_1}\;...\;\partial t_h^{i_h}}
    \Bigg|_{(0,...,0)}(x)=0,}\quad \text{if} \quad i_p=0 \quad
  \text{for some} \quad 1 \leq p\leq h$;
\item[\em{(ii)}]
  $\displaystyle{\frac{\partial^h(\psi_{(t_1,...,t_h)})^{-1}}{\partial
      t_1\;...\;\partial t_h} \Big|_{(0,...,0)}(x)=
    -\frac{\partial^h\psi_{(t_1,...,t_h)}}{\partial t_1\;...\;\partial
      t_h}\Big|_{(0,...,0)}}(x)$;
\item[\em{(iii)}]
  $\displaystyle{\frac{\partial^h\psi_{(t_1,...,t_h)}}{\partial
      t_1\;...\;\partial
      t_h}\Big|_{(0,...,0)}(x)=\frac{\partial\psi_{\sqrt[h]{t},\dots,\sqrt[h]{t})}}{\partial
      t}\big|_{t=0}(x)}$
\end{enumerate}
at any point $x\in U$.
\end{lemma}
\bp Assertions (i) and (ii) can be obtained by direct computation.  It
follows from (i) that $\frac{\partial^h\psi_{(t_1,...,t_h)}} {\partial
  t_1\;...\;\partial t_h}\Big|_{(0,...,0)}(x)$ is the first
non-necessarily vanishing derivative of the diffeomorphism family
$\{\psi_{(t_1,...,t_h)}\}$ at any point $x\in M$. Using
\begin{displaymath}
  \psi_{(t_1,\dots,t_k)}(x) = x + t_1\cdots t_k\left(X(x) + \omega
    (x,t_1,\dots,t_k)\right),
\end{displaymath}
where $\displaystyle \lim_{t_i\to 0} \omega (x,t_1,\dots,t_k) = 0$ we obtain, that
\begin{displaymath}
  \frac{\partial}{\partial t}\Big|_{t=0} \psi_{(\sqrt[k]{t}, \dots,
    \sqrt[k]{t})}(x) = \frac{\partial}{\partial t}\Big|_{t=0}\Big (x +
  t\big(X(x) + \omega (x,\sqrt[k]{t},\dots,\sqrt[k]{t})\big)\Big) =
  X(x),
\end{displaymath}
which proves (iii).\ep 
We remark that the assertion $(iii)$ means that any vector field strongly tangent to $H$ 
is tangent to $H$. \\
Now, we generalize a well-known relation between the commutator of vector fields 
and the commutator of their induced flows. 
\begin{lemma}\label{bra}
  Let $\{\phi_{(s_1,...,s_k)}\}$ and $\{\psi_{(t_1,...,t_l)}\}$ be
  ${\mathcal C}^{\infty}$-differentiable $k$-parameter, respectively
  \,$l$-parameter families of (local) diffeomorphisms defined on a
  neighbourhood $U\subset{\mathbb R}^n$. Assume that
  $\phi_{(s_1,\dots,s_k)}=\mathsf{Id}$, respectively
  $\psi_{(t_1,\dots,t_l)}=\mathsf{Id}$, if some of their variables
  equals $0$. Then the family of (local) diffeomorphisms
  $[\phi_{(s_1,...,s_k)},\psi_{(t_1,...,t_l)}]$ defined by the
  commutator of the group $\mathsf{Diff}^{\infty}(U)$ fulfills
  $[\phi_{(s_1,...,s_k)},\psi_{(t_1,...,t_l)}]=\mathsf{Id}$, if some
  of its variables equals $0$. Moreover
  \begin{displaymath}
    \frac{\partial^{k+l}[\phi_{(s_1...s_k)},\psi_{(t_1...t_l)}]}
    {\partial s_1\;...\;\partial s_k\;\partial t_1\;...\;\partial
      t_l} \Big|_{(0...0;
      0...0)}(x)=-\Bigg[\frac{\partial^k\phi_{(s_1...s_k)}}
    {\partial s_1\;...\;\partial
      s_k}\Big|_{(0...0)},\frac{\partial^l\psi_{(t_1...t_l)}}
    {\partial t_1\;...\; \partial t_l}\Big|_{(0...0)}\Bigg](x)
  \end{displaymath}
  at any point $x\in U$.
\end{lemma}
\begin{proof} 
  The group theoretical commutator
  $\big[\phi_{(s_1,...,s_k)},\psi_{(t_1,...,t_l)}\big]$ of the
  families of diffeomorphisms satisfies
  $[\phi_{(s_1,...,s_k)},\psi_{(t_1,...,t_l)}]=\mathsf{Id}$, if some
  of its variables equals $0$. Hence
  \begin{displaymath}
    \displaystyle\frac{\partial^{i_1+...+i_k+j_1+...+j_l}[\phi_{(s_1,...,s_k)},
      \psi_{(t_1,...,t_l)}]}{\partial s_1^{i_1}\;...\;\partial
      s_k^{i_k} \partial t_1^{j_1}\;...\;\partial
      t_l^{i_l}}\Big|_{(0,...,0;0,...,0)}=0,
  \end{displaymath}
  if $i_p\!=\!0$ or $j_q\!=\!0$ for some index $1\!\leq \!p\!\leq\! k$ or $1 \! \leq \!
  q\! \leq\! l$.  The families of diffeomorphisms
  $\{\phi_{(s_1,...,s_l)}\}$, $\{\psi_{(t_1,...,t_l)}\}$,
  $\{\phi_{(s_1,...,s_l)}^{-1}\}$ and $\{\psi_{(t_1,...,t_l)}^{-1}\}$
  are the constant family $\mathsf{Id}$, if some of their variables
  equals $0$. Hence one has
\begin{alignat}{1}
  \label{comm}
  & \frac{\partial^{k+l}[\phi_{(s_1...s_k)},\psi_{(t_1...t_l)}]}
  {\partial s_1\;...\;\partial s_k\;\partial t_1\;...\;\partial t_l}
  \Big|_{(0,...,0;\; 0,...,0)}(x)=
  \\
  \notag & =\frac{\partial^k}{\partial s_1...\partial
    s_k}\Big|_{(0...0)} \Bigg(\frac{\partial^l
    \Big(\phi_{(s_1...s_k)}^{-1} \! \circ \psi_{(t_1...t_l)}^{-1} \!
    \circ \phi_{(s_1...s_k)} \! \circ \psi_{(t_1...t_l)}(x)\Big)}
  {\partial t_1...\partial t_l}\Big|_{(0...0)}\Bigg)
  \\
  \notag & = \frac{\partial^k}{\partial s_1...\partial
    s_k}\Big|_{(0...0)}\left(
    d(\phi_{(s_1...s_k)}^{-1})_{\phi_{(s_1...s_k)}(x)}
    \frac{\partial^l\psi_{(t_1...t_l)}^{-1}} {\partial t_1...\partial
      t_l}\Bigg|_{(0,...,0)}
    \!\!\!\!\!\!(\phi_{(s_1...s_k)}(x))\right),
\end{alignat}
where $d\big(\phi_{(s_1,...,s_k)}^{-1}\big)_{\phi_{(s_1,...,s_k)}(x)}$
denotes the Jacobi operator of the map $\phi_{(s_1,...,s_k)}^{-1}$ at
the point $\phi_{(s_1,...,s_k)}(x)$.  Using the fact, that
$\{\phi_{(s_1,...,s_k)}\}$ is the constant family $\mathsf{Id}$, if
some of its variables equals $0$, and the relation
$d(\phi_{(0,...,0)}^{-1})_{\phi_{(s_1,...,s_k)}(x)}=\mathsf{Id}$, we
obtain that (\ref{comm}) can be written as
{\small
\begin{displaymath}
  d\Big(\frac{\partial^k\phi_{(s_1...s_k)}^{-1}}{\partial
    s_1...\partial s_k}\Big|_{(0...0)}\Big)_x
  \!\!  \frac{\partial^l\psi_{(t_1...t_l)}^{-1}(x)}{\partial
    t_1...\partial t_l}\Big|_{(0...0)} \!\! +
  d\Big(\frac{\partial^l\psi_{(t_1...t_l)}^{-1}}{\partial
    t_1...\partial t_l}\Big|_{(0...0)}\Big)_x
  \frac{\partial^k\phi_{(s_1...s_k)}(x)}{\partial s_1...\partial
    s_k}\Big|_{(0,...,0)}.
\end{displaymath}
}According to assertion (ii) of Lemma \ref{-1} the last formula gives  
{\small
\begin{displaymath}
  d\Big(\frac{\partial^k\phi_{(s_1...s_k)}}{\partial
    s_1\;...\ \partial s_k}\Big|_{(0...0)}\Big)_x
  \!\! \frac{\partial^l\psi_{(t_1...t_l)}(x)}{\partial t_1\;...\; \partial
    t_l}\Big|_{(0...0)}\!\! - d\Big(\frac{\partial^l
    \psi_{(t_1...t_l)}} {\partial t_1\;...\; \partial
    t_l}\Big|_{(0...0)} \Big)_x \frac{\partial^k
    \phi_{(s_1...s_k)}(x)} {\partial s_1\;...\;\partial
    s_k}\Big|_{(0...0)},
\end{displaymath}
}which is the Lie bracket of vector fields 
\begin{displaymath}
  \left[{{\frac{\partial^l\psi_{(t_1,...,t_l)}}{\partial
          t_1\;...\;\partial t_l}\Big|_{(0,...,0)}}, \;
      \frac{\partial^k\phi_{(s_1,...,s_k)}} {\partial
        s_1\;...\;\partial s_k}\Big|_{(0,...,0)}}\right]:U\to {\mathbb
    R}^n.
\end{displaymath}
\hphantom{.}\end{proof} 
\bl The Lie algebra $\mathfrak v$  has a basis
consisting of vector fields strongly tangent to the group $H$.  
\el \bp The
iterated Lie brackets of vector fields belonging to $\mathcal V$
linearly generate the vector space $\mathfrak v$. It
follows from Lemma \ref{bra} that iterated Lie brackets of vector
fields belonging to $\mathcal V$ are strongly tangent to the group
$H$. Hence $\mathfrak v$ is linearly generated by
vector fields strongly tangent to $H$. \ep 
\bl \label{lincom}
Linear combinations of vector fields tangent to $H$ are tangent to
$H$. \el \bp If $X$ and $Y$ are vector fields tangent to $H$ then
there exist ${\mathcal C}^1$-differenti\-able $1$-parameter families of
diffeomorphisms $\{\phi_t\in H\}$ and $\{\psi_t\in H\}$ such that
\[\phi_0 \!=\!\psi_0\!=\!\mathsf{Id}, \qquad  
\frac{\partial}{\partial t}\Big|_{t=0}\phi_t = X, \qquad
\frac{\partial}{\partial t}\Big|_{t=0}\psi_t = Y.\] Considering the
${\mathcal C}^1$-differentiable $1$-parameter families of
diffeomorphisms $\{\phi_t \circ \psi_t\}$ and $\{\phi_{ct}\}$ one has
\[X + Y = \frac{\partial}{\partial t}\Big|_{t=0}(\phi_t \circ
\psi_t),\quad \quad c\,X = \frac{\partial}{\partial
  t}\Big|_{t=0}\phi_{(c\,t)}, \quad \text{for all} \quad c\in {\mathbb
  R}^n,\] which proves the assertion.\ep Lemmas
\ref{-1}\;--\;\ref{lincom} prove  Theorem \ref{liealg}. \end{proof}

\section{Curvature algebra}

\begin{definition} 
  A vector field $\xi\in {\mathfrak X}({\I}_xM)$ on the indicatrix
  $\I_xM$ is called a \emph{curvature vector field} of the Finsler
  manifold $(M, \F)$ at $x\in M$, if there exists $X, Y\in T_xM$ such
  that $\xi = r_x(X,Y)$, where
  \begin{equation}
    \label{eq:r}
    r_x(X,Y)(y):=R_{(x,y)}(l_yX,l_yY)
  \end{equation}
  The Lie subalgebra
  \begin{math}
    {\mathfrak R}_x \! := \! \big\langle \, r_x(X,Y); \;X, Y \!\in\!
    T_xM \, \big\rangle
  \end{math}
  of ${\mathfrak X}({\I}_xM)$ generated by the curvature vector fields
  is called the \emph {curvature algebra} of the Finsler manifold $(M,
  \F)$ at the point $x\in M$.
\end{definition}
Since the Finsler metric is preserved by parallel translations, its
derivatives with respect to horizontal vector fields are identically
zero. Using (\ref{eq:R_1}) we obtain, that the derivative of the Finsler metric 
with respect to (\ref{eq:r}) vanishes, and hence 
\begin{displaymath}
  g_{(x,y)}\big(y,R_{(x,y)}(l_yX ,l_yY)\big)=0, \quad \text{for any}
  \quad y,X,Y\in T_xM
\end{displaymath} 
(c.f.\,\cite{Shen1}, eq.\,(10.9)). This means that the curvature vector
fields $\xi\!=\!r_x(X,Y)$ are tangent to the indicatrix.  In the sequel we
investigate the tangential properties of the curvature algebra to the
holonomy group of the canonical connection $\nabla $ of a Finsler
manifold.

\begin{proposition}
  \label{curvat} 
  Any curvature vector field at $x\!\in\! M$ is strongly tangent to
  the holonomy group $\mathsf{Hol}(x)$.
\end{proposition}
\bp Indeed, let us consider the curvature vector field
$r_x(X,Y)\!\in\!  \mathfrak X(\I_xM)$, $X,Y\!\in\!T_xM$ and let $\hat
X, \hat Y \! \in\!  \mathfrak X (M)$ be commuting vector fields
i.e.~$[\hat X,\hat Y]\!=\!0$ such that $\hat X_x\!=\!X$, $\hat
Y_x\!=\!Y$. By the geometric construction, the flows $\{\phi_t\}$ and
$\{\psi_s\}$ of the horizontal lifts $l(\hat X)$ and $l(\hat Y)$ are
fiber preserving diffeomorphisms of the bundle $\I M$ for any
$t\in\R$, corresponding to parallel translations along integral curves
of $\hat X$ and $\hat Y$ respectively. Then the commutator
\begin{displaymath}
  \theta_{t,s}=[\phi_t, \psi_s] =\phi_t^{-1} \circ \psi_s^{-1} \circ
  \phi_t \circ \psi_s : \quad \I M\to\I M
\end{displaymath}
is also a fiber preserving diffeomorphism of the bundle $\I M$ for
any $t,s\in\R$. Therefore for any $x\in M$ the restriction
\begin{displaymath}
  \theta_{t,s}(x)=\theta_{t,s}\big|_{\I_xM}:\I_xM\to\I_xM
\end{displaymath}
to the fiber $\I_xM$ is a 2-parameter $C^\infty$-differentiable family
of diffeomorphisms contained in the holonomy group $\mathsf{Hol}(x)$
such that
\begin{displaymath}
  \theta_{0,s}(x) = \mathsf{Id}, \qquad \theta_{t,0}(x) = \mathsf{Id},
  \qquad \text{and} \qquad \frac{\partial^2}{\partial t\partial s}\Big
  |_{t=0,s=0}\theta_{t,s}(x)=r_x(X,Y),
\end{displaymath}
which proves that the curvature vector field $r_x(X,Y)$ is strongly
tangent to the holonomy group $\mathsf{Hol}(x)$ and hence we obtain
the assertion.  \ep

\begin{theorem}
  The curvature algebra ${\mathfrak R}_x$ of a Finsler manifold $(M,
  \F)$ is tangent to the holonomy group $\mathsf{Hol}(x)$ for any
  $x\in M$.
\end{theorem}

\bp Since by Proposition \ref{curvat} the curvature vector fields are
strongly tangent to $\mathsf{Hol}(x)$ and the curvature algebra
${\mathfrak R}_x$ is algebraically generated by the curvature vector
fields, the assertion follows from Theorem \ref{liealg}.  \ep
\begin{proposition} 
  The curvature algebra ${\mathfrak R}_x$ of a Riemannian manifold
  $(M,g)$ at any point $x\in M$ is isomorphic to the linear Lie
  algebra over the vector space $T_xM$ generated by the curvature
  operators of $(M,g)$ at $x\in M$. 
\end{proposition} 
\bp The curvature tensor field of a Riemannian manifold given by the
equation (\ref{eq:R_1}) is linear with respect to $y\in T_xM$ and
hence
\begin{displaymath}
  R_{(x,y)}(\xi,\eta) = (R_x(\xi, \eta))^k_ly^l \frac{\partial}
  {\partial y^k} ,
\end{displaymath}
where $R_x(\xi, \eta))^k_l$ is the matrix of the curvature
operator $R_x(\xi, \eta)\colon T_xM\to T_xM$ with respect to the natural
basis
\begin{math}
  \big\{\frac{\partial}{\partial x^1}|_x, ... ,
  \frac{\partial}{\partial x^n}|_x\big\}.
\end{math}
Hence any curvature vector field
$r_x(\xi,\eta)(y)$ with $\xi,\eta\in T_xM$ has the shape
$r_x(\xi,\eta)(y) = (R_x(\xi,
\eta))^k_ly^l\frac{\partial}{\partial y^k}$.  It follows that the flow
of $r_x(\xi,\eta)(y)$ on the indicatrix ${\I}_xM$ generated by the
vector field $r_x(\xi,\eta)(y)$ is induced by the action of the linear
1-parameter group $\exp tR_x(\xi, \eta))$ on $T_xM$, which implies the
assertion.  \ep
\begin{remark} The curvature algebra of Finsler surfaces is one-dimensional.
\end{remark}
\bp For Finsler surfaces the curvature vector fields form a one-dimensional vector space and 
hence the generated Lie algebra is also one-dimensional. \ep

\section{Constant curvature}

Now, we consider a Finsler manifold $(M,\mathcal F)$ of non-zero constant  
curvature. In this case for any $x\in M$ the curvature 
vector field $r_x(X,Y)(y)$ has the shape (cf. (\ref{gorb}))
    \begin{displaymath}
      r(X,Y)(y) = c\left(\delta_j^ig_{km}(y)y^m - \delta_k^ig_{jm}(y)
        y^m \right) X^jY^k\frac{\partial}{\partial y^i}, \quad 0\neq
      c\in\R.
    \end{displaymath}
Putting $y_j = g_{jm}(y)y^m $ we can write $r(X,Y)(y) = c\left(\delta_j^i y_k - \delta_k^i y_j
       \right)X^jY^k\frac{\partial}{\partial y^i}.$ 
Any linear combination of curvature vector fields has the form 
\begin{displaymath}
  r(A)(y) = A^{jk}\left(\delta_j^i y_k - \delta_k^i y_j
  \right)\frac{\partial}{\partial y^i},
\end{displaymath}
where $A = A^{jk}\frac{\partial}{\partial
  x^j}\wedge\frac{\partial}{\partial x^k} \in T_xM\wedge T_xM$ is
arbitrary bivector at $x\in M$.  \bl Let $(M,\mathcal F)$ be a Finsler
manifold of non-zero constant curvature. The curvature algebra
${\mathfrak R}_x$ at any point $x\in M$ satisfies
\begin{equation} \dim {\mathfrak R}_x \ge\frac{n(n-1)}{2},\end{equation} 
where $n=\dim M$.  \el 
\bp Let us consider the curvature vector fields
$r_{jk}=r_x(\frac{\partial}{\partial y^j},\frac{\partial}{\partial y^k})(y)$ at a fixed point 
$x\in M$. If a linear combination 
\[A^{jk}r_{jk} = A^{jk}(\delta_j^i y_k - \delta_k^i y_j)\frac{\partial}{\partial y^i} = 
(A^{ik}y_k - A^{ji} y_j)\frac{\partial}{\partial y^i}= 2A^{ik}y_k\frac{\partial}{\partial y^i}\] 
of curvature vector fields $r_{jk}$ with constant coefficients 
$A^{jk} = -A^{kj}$ satisfies $A^{jk}r_{jk}=0$ for any $y\in T_xM$ 
then one has the linear equation $A^{ik}y_k=0$ for any fixed index $i$. Since the 
covector fields $y_1,\dots ,y_n$ are linearly independent we obtain $A^{jk} = 0$ 
for all $j,k\in \{1,\dots ,n\}$. It follows  
that the curvature vector fields $r_{jk}$ are linearly independent for any $j<k$ 
and hence $\dim {\mathfrak R}_x \geq \frac{n(n-1)}{2}$.
\ep
\begin{corollary}
  Let $(M,g)$ be a Riemannian manifold of non-zero constant curvature
  with $n=\dim M$. The curvature algebra ${\mathfrak R}_x$ at any
  point $x\in M$ is isomorphic to the orthogonal Lie algebra
  $\mathfrak o(n)$.
\end{corollary}
\begin{proof}
  The holonomy group of a Riemannian manifold is a subgroup of the
  orthogonal group $O(n)$ of the tangent space $T_xM$ and hence the
  curvature algebra ${\mathfrak R}_x$ is a sub\-algebra of the
  orthogonal Lie algebra ${\mathfrak o}(n)$.  Hence the previous
  assertion implies the corollary.
\end{proof}
\begin{theorem}
  Let $(M,\mathcal F)$ be a Finsler manifold of non-zero constant
  curvature with $n\!=\!\dim M > 2$. If the point $x\in M$ is not
  (semi-)Riemannian then the curvature algebra ${\mathfrak R}_x$ at
  $x\in M$ satisfies
  \begin{equation}
    \label{ineq} 
    \dim {\mathfrak R}_x > \frac{n(n-1)}{2}.
  \end{equation}
\end{theorem}
\begin{proof} 
  We assume $\dim {\mathfrak R}_x = \frac{n(n-1)}{2}$. For any
  constant skew-symmetric matrices \;$\{A^{jk}\}$\; and
  \;$\{B^{jk}\}$\; the Lie bracket of vector fields
  \;$A^{ik}y_k\frac{\partial}{\partial y^i}$\; and
  \;$B^{ik}y_k\frac{\partial}{\partial y^i}$\; has the shape
  \;$C^{ik}y_k\frac{\partial}{\partial y^i}$,\; where \;$\{C^{ik}\}$\;
  is a constant skew-symmetric matrix, too.  Using the homogeneity of
  $g_{hl}$ we obtain
\begin{equation}
  \label{der}
  \frac{\partial y_h}{\partial y^m} = \frac{\partial g_{hl}}{\partial
    y^m}\,y^l + g_{hm} = g_{hm}
\end{equation}
and hence  
\begin{alignat*}{1}
  & \left[A^{mk}\,y_k\,\frac{\partial}{\partial
      y^m},B^{ih}\,y_h\,\frac{\partial}{\partial y^i}\right] = \left
    (A^{mk}\,B^{ih}\,\frac{\partial y_h}{\partial y^m} -
    B^{mk}\,A^{ih}\,\frac{\partial y_h}{\partial y^m}\right
  )y_k\,\frac{\partial}{\partial y^i}
  \\
  & \qquad = \left (B^{ih}\,g_{hm}\,A^{mk} -
    A^{ih}\,g_{hm}\,B^{mk}\right )y_k\,\frac{\partial}{\partial y^i} =
  C^{ik}\,y_k\,\frac{\partial}{\partial y^i}.
\end{alignat*}
Particularly, for the skew-symmetric matrices $E^{ij}_{ab}=\delta
  ^i_a \delta ^j_b \! -\! \delta ^i_b \delta ^j_a$, $a,b\in
\{1,\dots ,n\}$, we have
\begin{displaymath}
\left[E^{ij}_{ab}\,y_j\,\frac{\partial}{\partial y^i},
  \,E^{kl}_{cd}\, y_l\,\frac{\partial}{\partial y^k}\right] = 
\left (E^{ih}_{cd}\,g_{hm}\,E^{mk}_{ab} -
  E^{ih}_{ab}\,g_{hm}\,E^{mk}_{cd}\right
)y_k\,\frac{\partial}{\partial y^i} = \Lambda
^{im}_{ab,cd}\,y_m\,\frac{\partial}{\partial y^i},
\end{displaymath}
where the constants $\Lambda ^{ij}_{ab,cd}$ satisfy 
\begin{math}
  \Lambda ^{ij}_{ab,cd}\!=\!-\Lambda ^{ji}_{ab,cd}\!=\!-\Lambda
  ^{ij}_{ba,cd} \!=\!-\Lambda ^{ij}_{ab,dc}\!=\!-\Lambda
  ^{ij}_{cd,ab}.
\end{math}
Putting $i=a$ and computing the trace for these indices we obtain
\begin{equation}\label{lambda}
  (n-2)(g_{bd}\,y_c - g_{bc}\,y_d) = \Lambda ^{l}_{b,cd}\,y_l,
\end{equation}
where $\Lambda ^{l}_{b,cd} := \Lambda ^{il}_{ib,cd}$. The right hand
side is a linear form in variables $y_1,\dots,y_n$. According to the
identity (\ref{lambda}) this linear form vanishes for $y_c=y_d =0$,
hence $\Lambda ^{l}_{b,cd} = 0$ for $l\neq c,d$. Denoting
$\lambda_{bd}^{(c)}:= \frac{1}{n-2}\Lambda ^{c}_{b,cd}$ (no summation
for the index $c$) we get the identities
\[g_{bd}\,y_{c} - g_{bc}\,y_d = \lambda ^{(c)}_{bd}\,y_{c} - \lambda
^{(d)}_{bc}\,y_{d}\quad \text{(no summation for $c$ and $d$)}.\]
Putting $y_d = 0$ we obtain $g_{bd}\big|_{y_d=0} = \lambda
^{(c)}_{bd}$ for any $c\neq d$. It follows $\lambda ^{(c)}_{bd}$ is
independent of the index $c\;(\neq d)$.  Defining $\lambda _{bd}
:=\lambda ^{(c)}_{bd}$ with some $c\;(\neq d)$ we obtain from
(\ref{lambda}) the identity
\begin{equation}
  \label{vegso} 
  g_{bd}\,y_{c} - g_{bc}\,y_d = \lambda _{bd}\,y_{c} - \lambda
  _{bc}\,y_{d}
\end{equation} 
for any $b,c,d\in \{1,\dots,n\}$.  We have
\begin{displaymath}
  \lambda _{cd}\,y_{b} - \lambda _{cb}\,y_{d} = (g_{bd}\,y_{c} -
  g_{bc}\,y_d) - (g_{db}\,y_{c} - g_{dc}\,y_b) = (\lambda _{bd}\,y_{c}
  - \lambda _{bc}\,y_d) - (\lambda _{db}\,y_{c} - \lambda _{dc}\,y_b).
\end{displaymath}
which implies the identity
\begin{alignat}{1}
  \notag & (\lambda _{cd}\,y_{b} - \lambda _{cb}\,y_{d}) + (\lambda
  _{db}\,y_{c} - \lambda _{dc}\,y_b) + (\lambda _{bc}\,y_d - \lambda
  _{bd}\,y_{c}) =
  \\
  \label{szimm}
  & \qquad = (\lambda _{cd} - \lambda _{dc})\,y_b + (\lambda _{db} -
  \lambda _{bd})\,y_{c} + (\lambda _{bc} - \lambda _{cb})\,y_{d} = 0.
\end{alignat}
Since $\dim M > 2$, we can consider $3$ different indices $b, c, d$
and we obtain from the identity (\ref{szimm}) that $\lambda _{bc} =
\lambda _{cb}$ for any $b, c\in \{1,\dots ,n\}$.  \\[1ex]\indent By
derivation the identity (\ref{vegso}) we get
\begin{displaymath}
  \frac{\partial g_{bd}}{\partial y_a}\,y_{c} - \frac{\partial
    g_{bc}}{\partial y_a}\,y_d + g_{bd}\,\delta ^{a}_{c} -
  g_{bc}\,\delta ^{a}_d = \lambda _{bd}\,\delta ^{a}_{c} - \lambda
  _{bc}\,\delta ^{a}_{d}.
\end{displaymath}
Using (\ref{der}) we obtain
\begin{alignat*}{1}
  \frac{\partial y_a}{\partial y^q}\left(\frac{\partial
      g_{bd}}{\partial y_a}\,y_{c} - \frac{\partial g_{bc}}{\partial
      y_a}\,y_d\right) + g_{bd}\,g_{cq} - g_{bc}\,g_{dq} & =
  \\
  = \frac{\partial g_{bd}}{\partial y^q}\,y_{c} - \frac{\partial
    g_{bc}}{\partial y^q}\,y_d + g_{bd}\,g_{cq} - g_{bc}\,g_{dq} & =
  \lambda _{bd}\,g_{cq} - \lambda _{bc}\,g_{dq}.
\end{alignat*}
Since
\begin{displaymath}
  \left(\frac{\partial g_{bd}}{\partial y^q}\,y_{c} - \frac{\partial
      g_{bc}}{\partial y^q}\,y_d\right)y^b = 0\] we get the identity
  \[y_{d}\,g_{cq} - y_{c}\,g_{dq} = \lambda _{bd}\,y^b\,g_{cq} -
  \lambda _{bc}\,y^b\,g_{dq}.
\end{displaymath}
Multiplying both sides of this identity by the inverse $\{g^{qr}\}$ of
the matrix $\{g_{cq}\}$ and taking the trace with respect to the
indices $c,r$ we obtain the identity
\[(n-1)\,y_{d} = (n-1)\lambda _{bd}\,y^b.\] Hence we obtain that
$g_{bd}\,y^b = \lambda _{bd}\,y^b$ and hence $g_{bd} = \lambda _{bd},$
which means that the point $x\in M$ is (semi-)Riemannian. From this
contradiction follows the assertion.
\end{proof}
\bt Let $(M,\mathcal F)$ be a positive definite Finsler manifold of
non-zero constant curvature with $n=\dim M >2$. The holonomy group of
$(M,\mathcal F)$ is a compact Lie group if and only if $(M,\mathcal
F)$ is a Riemannian manifold.  \et \bp We assume that the holonomy
group of a Finsler manifold $(M,\mathcal F)$ of non-zero constant
curvature with $\dim M \geq 3$ is a compact Lie transformation group
on the indicatrix ${\I}_xM$.  The curvature algebra ${\mathfrak R}_x$
at a point $x\in M$ is tangent to the holonomy group $\mathsf{Hol}(x)$
and hence $\dim \mathsf{Hol}(x) \ge \dim{\mathfrak R}_x $.  If there
exists a not (semi-)Riemannian point $x\in M$ then $\dim {\mathfrak
  R}_x > \frac{n(n-1)}{2}$.  The $(n-1)$-dimensional indicatrix
${\I}_xM$ at $x$ can be equipped with a Riemannian metric which is
invariant with respect to the compact Lie transformation group
$\mathsf{Hol}(x)$. Since the group of isometries of an
$n-1$-dimensional Riemannian manifold is of dimension at most
$\frac{n(n-1)}{2}$ (cf. Kobayashi \cite{Kob}, p.~46,) we obtain a
contradiction, which proves the assertion.  \ep Since the holonomy
group of a Landsberg manifold is a subgroup of the isometry group of
the indicatrix, we obtain that any Landsberg manifold of non-zero
constant curvature with dimension $> 2$ is Riemannian (c.f. Numata
\cite{Num}).  \\[1ex]\indent We can summarize our results as follows:
\bt The holonomy group of any non-Riemannian positive definite Finsler
manifold of non-zero constant curvature with dimension $> 2$ does not
occur as the holonomy group of any Riemannian manifold.  \et

\section{Appendix: Finsler metric with infinite dimensional curvature
  algebra}

Let us consider the singular (non $y$-global) Finsler manifold $(H_3, \mathcal F)$,
where $H_3$ is the 3-dimensional Heisenberg group and $\mathcal F$ is
a left-invariant Berwald-Mo\'or metric (c.f. \cite{Shen1}, Example
1.1.5, p.~8).
\\
The group $H_3$ can be realized as the Lie group of matrices of the
form
\begin{math}\left[
    \begin{smallmatrix} 
      1 & x^1 & x^2
      \\
      0 & 1 & x^3
      \\
      0 & 0 & 1
    \end{smallmatrix} \right],\end{math} where $x=(x^1,x^2,x^3)\in
\mathbb{R}^3$ and hence the multiplication can be written as
\begin{displaymath}
  (x^1,x^2,x^3)\cdot(y^1,y^2,y^3)=(x^1+y^1,x^2+y^2+x^1y^3,x^3+y^3).
\end{displaymath}
The vector $0\!=\!(0,0,0)\in \mathbb{R}^3$ gives the unit element of
$H_3$. The Lie algebra ${\mathfrak h}_3=T_0H_3$ consists of matrices
of the form
\begin{math}
  \left[
    \begin{smallmatrix}
      0 & a^1 & a^2
      \\
      0 & 0 & a^3
      \\
      0 & 0 & 0
    \end{smallmatrix}
  \right],
\end{math}
corresponding to the tangent vector
\begin{math}
  a\! = \! a^1\frac{\partial}{\partial x^1} +
  a^2\frac{\partial}{\partial x^2} + a^3\frac{\partial}{\partial x^3}
\end{math}
at the unit element $0\in H_3$.  A left-invariant Berwald-Mo\'or
Finsler metric $\mathcal F$ is induced by the (singular) Minkowski
functional $\mathcal F_{_0}\colon {\mathfrak h}_3 \to \mathbb R$:
\begin{displaymath}
  \mathcal F_{_0}(a):= \left( a^1a^2a^3 \right)^{\frac{2}{3}}
\end{displaymath}
of the Lie algebra in the following way: if $y=(y^1,y^2,y^3)$ is a
tangent vector at $x \in H_3$, then
\[ \mathcal F(x,y) := \mathcal F_{_0}(x^{-1}y).\] The coordinate
expression of the singular (non $y$-global) Finsler metric $\mathcal F$ is
\begin{displaymath}
  \mathcal F(x,y) = \left(y^1 \left(y^2\!-\!x^1y^3 \right) 
    y^3\right)^{\frac{2}{3}}.
\end{displaymath}
Since $\mathcal F$ is left-invariant, the associated geometric
structures (connection, geodesics, curvature) are also left-invariant
and the curvature algebras at different points are isomorphic.  Using
the notation
\[r_x(i,j)=r_x\Bigl(\frac{\partial}{\partial x^i},\frac{\partial}{\partial
  x^j}\Bigl), \quad i,j=1,2,3,\] 
for curvature vector fields, a direct computation yields  
\vspace{-22pt}
\begin{center}
 \small
   \begin{alignat*}{1}
     r_x(1,2)& = \frac{1}{4}\left( {\frac
         {5{y^1}^{2}{y^3}^{2}}{\left(x^1y^3 \! - \!  y^2 \right)
           ^{3}}}\; \frac{\partial}{\partial y^1} \,+ \, {\frac
         {y^1{y^3}^{2} \left( 3\,x^1y^3+y^2 \right) }{ \left( y^2 \!
             - \! x^1y^3 \right) ^{3}}}\; \frac{\partial}{\partial
         y^2} \,+ \, {\frac {4y^1{y^3}^{3}}{ \left(y^2 \! - \! x^1y^3
           \right) ^{3}}}\; \frac{\partial}{\partial y^3} \right),
     \\
     r_x(1,3)& = \frac{1}{4} \Bigg(\frac {{y^1}^{2}y^3 \left( 6\,
         x^1y^3 \! - \! 11\,y^2 \right)} {\left( x^1y^3 \! - \! y^2
       \right) ^{3}}\; \frac{\partial}{\partial y^1} \,+ \, {\frac {4
         y^1{y^3}^{2}x^1 \left( 2\,x^1y^3 \! - \! 3\,y^2 \right) }{
         \left(
           y^2 \! - \! x^1y^3 \right) ^{3}}}\; \frac{\partial}{\partial y^2} \,+ \\
     & \quad \quad \quad + {\frac {y^1{y^3}^{2} \left( 7\,x^1y^3 \! -
           \!  11\,y^2 \right) }{ \left( y^2 \! - \!  x^1y^3 \right)
         ^{3}}}\; \frac{\partial}{\partial y^3} \Bigg),
     \\
     r_x(2,3) &= \frac{1}{4} \left( {\frac{4 {y^1}^{3}y_{{3}}} {
           \left( x^1y^3 \! - \!  y^2 \right)^{3}}}\;
       \frac{\partial}{\partial y^1} \,+ \, {\frac {{y^1}^{2}y^3
           \left( 6\,x^1y^3 \! - \!  y^2 \right) }{ \left( y^2 \! - \!
             x^1y^3 \right) ^{3}}}\; \frac{\partial}{\partial y^2} \,+
       \, {\frac {5 {y^1}^{2}{y^3}^{2}}{ \left( y^2 \! - \!  x^1y^3
           \right) ^{ 3}}}\; \frac{\partial}{\partial y^3}\right).
    \end{alignat*}
  \end{center}
  The curvature vector fields $r_0(i,j), \; i,j=1,2,3,$ at the unit
  element $0\in H_3$ generate the curvature algebra ${\mathfrak
    r}_0$. Let us denote \( Y^{k,m}:=
  \frac{{y^1}^{k}{y^3}^{m}}{{y^2}^{k+m-1}}, \;k,m \in \mathbb N\), and
  consider the vector fields
\begin{equation}
  \label{eq:A}
  A^{k,m}(a^1,a^2,a^3)= a^1 Y^{k+1,m}\frac{\partial}{\partial
    y^1}\Big|_0 \,+ \,a^2 Y^{k,m}\frac{\partial}{\partial y^2}\Big|_0
  \,+ \,a^3 Y^{k,m+1}\frac{\partial}{\partial y^3}\Big|_0,
\end{equation}
with $(a^1,a^2,a^3) \in \mathbb R^3$ and $k,m \in \mathbb N$. Then the
curvature vector fields $r_0(i,j)$ at $0\in H_3$ can be written in the
form
{\small
  \begin{displaymath}
    r_0(1,2)=\frac 14 A^{1,2}(-5,1,4), \; \;r_0(1,3)=\frac 14
    A^{1,1}(11,0,-11),\; \; r_0(2,3)=\frac 14 A^{2,1}(-4,-1,5).
  \end{displaymath}}
\begin{proposition}
  The curvature algebra ${\mathfrak r}_x$ at any point $x\in M$ is a
  Lie algebra of infinite dimension.
\end{proposition}
\bp Since the Finsler metric is left-invariant, the curvature algebras
at different points are isomorphic.  Therefore it is enough to prove
that the curvature algebra ${\mathfrak r}_0$ at $0\! \in \!H_3$ has
infinite dimension.  We prove the statement by contradiction: let us
suppose that ${\mathfrak r}_0$ is finite dimensional.
\\
A direct computation shows that for any $(a^1,a^2,a^3),
(b^1,b^2,b^3)\in \mathbb R^3$ one has
\[\left[A^{k,m}(a^1,a^2,a^3),A^{p,q}(b^1,b^2,b^3)\right ]=A^{k+p,m+q}(c^1,c^2,c^3)\] 
with some $(c^1,c^2,c^3)\in \mathbb R^3$.  It follows that any
iterated Lie bracket of curvature vector fields $r_0(i,j), \;
i,j=1,2,3,$ has the shape (\ref{eq:A}) and hence there exists a basis
of the curvature algebra ${\mathfrak r}_0$ of the form
$\{A^{k_i,m_i}(a_i^1,a_i^2,a_i^3)\}_{i=1}^{N}$, where $N\in \mathbb N$
is the dimension of ${\mathfrak r}_0$.  We can assume that
$\{(k_i,m_i)\}_{i=1}^N$ forms an increasing sequence,
i.e. $(k_1,m_1)\le(k_2,m_2)\le\dots\le(k_N,m_N)$ holds with respect to
the lexicographical ordering of $\mathbb N\times \mathbb N$. We can
consider the vector fields
\begin{displaymath}
  \frac{4}{11} r_0(1,3)=A^{1,1}(1,0,-1),\; 4 r_0(1,2)=A^{1,2}(-5,1,4),
  \; 4 r_0(2,3)=A^{2,1}(-4,-1,5)
\end{displaymath}
as the first three members of this sequence.  Hence $1\leq k_N,m_N$
and
\begin{displaymath}
  \left[A^{1,1}(1,0,-1),\ A^{k_N,m_N}(a_N^1,a_N^2,a_N^3)\right ] =
  A^{1+k_N,1+m_N}(c^1,c^2,c^3)
\end{displaymath}
belongs to ${\mathfrak r}_0$, too, where $c^1=(k_N-m_N-1)a_N^1 +
2a_N^2 - a_N^3$, $c^2=(k_N-m_N)a_N^2$ and $c^3=a_N^1 - 2a_N^2 +
(k_N-m_N+1)a_N^3$. Since $k_N< 1+k_N$, $m_N< 1+m_N$ we have
$c^1=c^2=c^3=0$ and hence the homogeneous linear system 
\begin{alignat*}{1}
  0  & = (k_N-m_N-1)a_N^1 + 2a_N^2 - a_N^3,
  \\
  0 & = (k_N-m_N)a_N^2,
  \\
  0 &= a_N^1 - 2a_N^2 + (k_N-m_N+1)a_N^3 
\end{alignat*}
has a solution $(a_N^1,a_N^2,a_N^3)\ne(0,0,0)$. It follows that $k_N=m_N$. \\
Similarly, computing the Lie bracket
\[0=\left[A^{1,2}(-5,1,4),A^{k_N,k_N}(a_N^1,a_N^2,a_N^3)\right ] =
A^{1+k_N,2+k_N}(d^1,d^2,d^3)\] Since $k_N< 1+k_N<2+k_N$ we have
$d^1=d^2=d^3=0$ giving the homogeneous linear system 
\begin{alignat*}{1}
  0 = & (-3k_N+5)a_N^1 - 15a_N^2 + 10a_N^3,
  \\
  0 = & - a_N^1+(3-3k_N)a_N^2 - 2a_N^3,
  \\
  0 = & -4a_N^1 + 12a_N^2 - (3k_N + 8)a_N^3
\end{alignat*} for $(a_N^1,a_N^2,a_N^3)$. The determinant of this
system vanishes only for $k_N=0$ which is a contradiction.\ep \bc The
holonomy group of the Finsler manifold $(H_3, \mathcal F)$ has an
infinite dimensional tangent Lie algebra.  \ec We remark here, that it
remains an interesting open question: {\it Is there a nonsingular
  ($y$-global) Finsler manifold whose holonomy group is infinite
  dimensional?}

\end{document}